\documentclass[11pt]{article}
\usepackage{amsfonts,amssymb,amsmath,amscd,latexsym,makeidx,theorem,graphicx}

\author{Luca Martinazzi  \\ \footnotesize Rutgers University \\ \footnotesize \texttt{luca.martinazzi@math.rutgers.edu}}
\title{Conformal metrics on $\R{2m}$ with constant $Q$-curvature and large volume}
\date{January 20, 2012}

\newtheorem{trm}{Theorem}

\newtheorem{lemma}[trm]{Lemma}

\newcommand{\R}[1]{\mathbb{R}^{#1}}

\newcommand{\de}{\partial}
\newcommand{\ve}{\varepsilon}

\newenvironment{proof}{\noindent\emph{Proof.}}{\hfill$\square$\medskip}

\DeclareMathOperator{\loc}{loc}

\DeclareMathOperator{\vol}{vol}

\begin{document}
\maketitle

\begin{abstract}
We study conformal metrics $g_u=e^{2u}|dx|^2$ on $\R{2m}$ with constant $Q$-curvature $Q_{g_u}\equiv (2m-1)!$ (notice that $(2m-1)!$ is the $Q$-curvature of $S^{2m}$) and finite volume. When $m=3$ we show that there exists $V^*$ such that for any $V\in [V^*,\infty)$ there is a conformal metric $g_u=e^{2u}|dx|^2$ on $\R{6}$ with $Q_{g_u}\equiv 5!$ and $\vol(g_u)=V$. This is in sharp contrast with the four-dimensional case, treated by C-S. Lin. We also prove that when $m$ is odd and greater than $1$, there is a constant $V_m>\vol (S^{2m})$ such that for every $V\in (0,V_m]$ there is a conformal metric $g_u=e^{2u}|dx|^2$ on $\R{2m}$ with $Q_{g_u}\equiv (2m-1)!$,  $\vol(g)=V$. This extends a result of A. Chang and W-X. Chen. When $m$ is even we prove a similar result for conformal metrics of \emph{negative} $Q$-curvature.
\end{abstract}

\textsc{Keywords:} $Q$-curvature, Paneitz operators, GMJS operators, conformal geometry.

\section{Introduction and statement of the main theorems}

We consider solutions to the equation
\begin{equation}\label{eq0}
(-\Delta)^m u=(2m-1)! e^{2mu} \quad \textrm{in }\R{2m},
\end{equation}
satisfying
\begin{equation}\label{area}
V:=\int_{\R{2m}} e^{2mu(x)}dx<+\infty,
\end{equation}
with particular emphasis on the role played by $V$.

Geometrically, if $u$ solves \eqref{eq0} and \eqref{area}, then the conformal metric $g_u:=e^{2u}|dx|^2$ has $Q$-curvature $Q_{g_u}\equiv  (2m-1)!$ and volume $V$ (by $|dx|^2$ we denote the Euclidean metric). For the definition of $Q$-curvature and related remarks, we refer to Chapter 4 in \cite{cha} or to \cite{FG} and \cite{FH}.
Notice that given a solution $u$ to \eqref{eq0} and $\lambda>0$, the function $v:=u-\frac{1}{2m}\log \lambda$ solves
$$(-\Delta)^m v=\lambda (2m-1)! e^{2mv} \quad \textrm{in }\R{2m},\quad \int_{\R{2m}}e^{2mv(x)}dx=\frac V \lambda,$$
hence there is no loss of generality in the particular choice of the constant $(2m-1)!$ in \eqref{eq0}. On the other hand this constant has the advantage of being the $Q$-curvature of the round sphere $S^{2m}$. This implies that the function $u_1(x)=\log\frac{2}{1+|x|^2}$, which satisfies  $e^{2u_1}|dx|^2=(\pi^{-1})^*g_{S^{2m}}$ (here $\pi:S^{2m}\to\R{2m}$ is the stereographic projection) is a solution to \eqref{eq0}-\eqref{area} with $V=\vol(S^{2m})$. 
Translations and dilations (i.e. M\"obius transformations) actually give us a large family of solutions to \eqref{eq0}-\eqref{area} with $V=\vol(S^{2m})$, namely
\begin{equation}\label{stand}
u_{x_0,\lambda}(x):=u_1(\lambda(x-x_0))+\log\lambda =\log\frac{2\lambda}{1+\lambda^2|x-x_0|^2},\quad x_0\in \R{2m},\lambda>0.
\end{equation}
We shall call the functions $u_{x_0,\lambda}$ standard or \emph{spherical} solutions to \eqref{eq0}-\eqref{area}.

\medskip

The question whether the family of spherical solutions in \eqref{stand} exhausts the set of solutions to \eqref{eq0}-\eqref{area} has raised a lot of interest and is by now well understood. W. Chen and C. Li \cite{CL} proved that on $\R{2}$ ($m=1$) every solution to \eqref{eq0}-\eqref{area} is spherical, while for every $m>1$, i.e. in dimension $4$ and higher, it was proven by A. Chang and W-X. Chen \cite{CC} that Problem \eqref{eq0}-\eqref{area} admits solutions which are non spherical. In fact they proved 

\vspace{5mm}
\noindent{\bf Theorem A (A. Chang-W-X. Chen \cite{CC} 2001)}. {\it For every $m>1$ and $V\in (0,\vol(S^{2m}))$
there exists a solution to \eqref{eq0}-\eqref{area}.}
\vspace{5mm}

Several authors have tried to classify spherical solutions or, in other words, to give analytical and geometric conditions under which a solution to \eqref{eq0}-\eqref{area} is spherical (see \cite{CY}, \cite{WX}, \cite{Xu}), and to understand some properties of non-spherical solutions, such as their asymptotic behavior, their volume and their symmetry (see \cite{Lin}, \cite{mar1}, \cite{WY}). In particular C-S. Lin proved:

\vspace{5mm}
\noindent{\bf Theorem B (C-S. Lin \cite{Lin} 1998)}. {\it Let $u$ solve \eqref{eq0}-\eqref{area} with $m=2$. Then either $u$ is spherical (i.e. as in \eqref{stand}) or $V<\vol(S^{4}).$ }
\vspace{5mm}

Both spherical solutions and the solutions given by Theorem A are radially symmetric (i.e. of the form $u(|x-x_0|)$ for some $x_0\in \R{2m}$). On the other hand there also exist plenty of non-radial solutions to \eqref{eq0}-\eqref{area} when $m=2$.

\vspace{5mm}
\noindent{\bf Theorem C (J. Wei and D. Ye \cite{WY} 2006)}. {\it For every $V\in (0,\vol(S^{4}))$ there exist (several) non-radial solutions to \eqref{eq0}-\eqref{area} for $m=2$.}
\vspace{5mm}

\noindent\textbf{Remark D} Probably the proof of Theorem C can be extended to higher dimension $2m\ge 2$,   yielding several non-symmetric solutions to \eqref{eq0}-\eqref{area} for every $V\in (0,\vol(S^{2m}))$, but failing to produce non-symmetric solutions for $V\ge \vol(S^{2m})$. As in the proof of Theorem A, the condition $V<\vol(S^{2m})$ plays a crucial role.
\vspace{5mm}

Theorems A, B, C and Remark D strongly suggest that also in dimension $6$ and higher all non-spherical solutions to \eqref{eq0}-\eqref{area} satisfy $V<\vol(S^{2m})$, i.e.  \eqref{eq0}-\eqref{area} has no solution for $V>\vol(S^{2m})$ and the only solutions with $V=\vol(S^{2m})$ are the spherical ones. Quite surprisingly we found out that this is not at all the case. In fact in dimension $6$ we found solutions to \eqref{eq0}-\eqref{area} with arbitrarily large $V$:

\begin{trm}\label{trm0} For $m=3$ there exists $V^*>0$ such that for every $V\ge V^*$ there is a solution $u$ to \eqref{eq0}-\eqref{area}, i.e. there exists a metric on $\R{6}$ of the form $g_u=e^{2u}|dx|^2$ satisfying $Q_{g_u}\equiv 5!$ and $\vol(g_u)=V.$
\end{trm}

In order to prove Theorem \ref{trm0} we will consider only rotationally symmetric solutions to \eqref{eq0}-\eqref{area}, so that \eqref{eq0} reduces to and ODE. Precisely, given $a,b\in \R{}$ let $u=u_{a,b}(r)$ be the solution of
\begin{equation}\label{ODE}
\left\{
\begin{array}{l}
\Delta^3 u= -e^{6u}\quad \text{in }\R{6}\\
u(0)=u'(0)=u'''(0)=u'''''(0)=0 \rule{0cm}{.5cm}\\
u''(0)=\frac{\Delta u(0)}{6}=a \rule{0cm}{.5cm}\\
u''''(0)=\frac{\Delta^2u(0)}{16}=b.\rule{0cm}{.5cm}
\end{array}
\right.
\end{equation}
Here and in the following we will always (by a little abuse of notation) see a rotationally symmetric function $f$ both as a function of one variable $r\in [0,\infty)$ (when writing $f'$, $f''$, etc...) and as a function of $x\in \R{6}$ (when writing $\Delta f$, $\Delta^2 f$, etc...). We also used that 
$$\Delta f(0)=6f''(0),\quad \Delta^2f(0)=16f''''(0),$$
see e.g. \cite[Lemma 17]{mar1}. Also notice that in \eqref{ODE} we replaced $5!$ by $1$ to make the computations lighter. As we already noticed, this is not a problem.

\begin{trm}\label{trm1} Let $u=u_{a,3}$ solve \eqref{ODE} for a given $a<0$ and $b=3$.\footnote{The choice $b=3$ is convenient in the computations, but any other $b>0$ would work.} Then
\begin{equation}\label{lima}
\int_{\R{6}}e^{6u_{a,3}}dx<\infty\text{ for $-a$ large;}\quad  \lim_{a\to-\infty}\int_{\R{6}}e^{6u_{a,3}}dx=\infty.
\end{equation}
In particular the conformal metric $g_{u_{a,3}}=e^{2u_{a,3}}|dx|^2$ of constant $Q$-curvature $Q_{g_{u_{a,3}}}\equiv 1$ satisfies $$\lim_{a\to-\infty} \mathrm{vol}(g_{u_{a,3}})=\infty.$$
\end{trm}

Theorem \ref{trm0} will follow from Theorem \ref{trm1} and a continuity argument (Lemma \ref{cont} below).

Going through the proof of Theorem A it is clear that it does not extend to the case $V > \vol(S^{2m})$. With a different approach, we are able to prove that, at least when $m\ge 3$ is odd, one can extend Theorem A as follows.
\begin{trm}\label{trm3} For every $m\ge 3$ odd there exists $V_m>\vol(S^{2m})$ such that for every $V\in (0,V_m]$ there is a non-spherical solution $u$ to \eqref{eq0}-\eqref{area}, i.e. there exists a metric on $\R{2m}$ of the form $g_u=e^{2u}|dx|^2$ satisfying
$Q_{g_u}\equiv (2m-1)!$ and $\vol(g_u)=V.$
\end{trm}

The condition $m\ge 3$ odd is (at least in part) necessary  in view of Theorem B and \cite{CL}, but the case $m\ge 4$ even is open.
Notice also that when $m=3$, Theorems \ref{trm0} and \ref{trm3} guarantee the existence of solutions to \eqref{eq0}-\eqref{area} for
$$V \in (0,V_m]\cup  [V^*,\infty),$$
but we cannot rule out that $V_m<V^*$ (the explicit value of $V_m$ is given in \eqref{Vm0} below) and the existence of solutions to \eqref{eq0}-\eqref{area} is unknown for $V\in (V_m, V^*)$. Could there be a gap phenomenon?

\medskip

We now briefly investigate how large the volume of a metric $g_u=e^{2u}|dx|^2$ on $\R{2m}$ can be when $Q_{g_u}\equiv const<0$. Again with no loss of generality we assume $Q_{g_u}\equiv -(2m-1)!$.  In other words consider the problem
\begin{equation}\label{eq0-}
(-\Delta)^m u=-(2m-1)! e^{2mu} \quad \textrm{on }\R{2m}.
\end{equation}
Although for $m=1$ it is easy to see that Problem \eqref{eq0-}-\eqref{area} admits no solutions for any $V>0$, when $m\ge 2$ Problem \eqref{eq0-}-\eqref{area} has solutions for some $V>0$, as shown in \cite{mar2}. Then with the same proof of Theorem \ref{trm3} we get:

\begin{trm}\label{trm4} For every $m\ge 2$ even there exists $V_m>\vol(S^{2m})$ such that for $V\in (0,V_m]$ there is a solution $u$ to \eqref{eq0-}-\eqref{area}, i.e. there exists a metric on $\R{2m}$ of the form $g_u=e^{2u}|dx|^2$ satisfying
$$Q_{g_u}\equiv -(2m-1)!, \quad \vol(g_u)=V.$$
\end{trm}

The cases of solutions to \eqref{eq0}-\eqref{area} with $m$ even, or \eqref{eq0-}-\eqref{area} and $m$ odd  seem more difficult to treat since the ODE corresponding to \eqref{eq0} or \eqref{eq0-}, in analogy with \eqref{ODE} becomes
$$\Delta^m u(r)=(2m-1)! e^{2m u(r)},$$
whose solutions can blow up in finite time (i.e. for finite $r$) if the initial data are not chosen carefully (contrary to Lemma \ref{exist} below).

\section{Proof of Theorem \ref{trm1}}
Set $\omega_{2m-1}:=\vol(S^{2m-1})$ and let $B_r$ denote the unit ball in $\R{2m}$ centered at the origin.
Given a smooth radial function $f=f(r)$ in $\R{2m}$ we will often use the divergence theorem in the form
\begin{equation}\label{div1}
\int_{B_r}\Delta f dx= \int_{\de B_r} \frac{\de f}{\de\nu}d\sigma=\omega_{2m-1} r^{2m-1} f'(r).
\end{equation}
Dividing by $\omega_{2m-1} r^{2m-1}$ into \eqref{div1} and integrating we also obtain
\begin{equation}\label{div2}
f(t)-f(s)=\int_s^t\frac{1}{\omega_{2m-1} \rho^{2m-1}}\int_{B_\rho}\Delta f dxd\rho,\quad 0\le s\le t.
\end{equation}

When no confusion can arise we will simply write $u$ instead of $u_{a,3}$ or $u_{a,b}$ to denote the solution to \eqref{ODE}. In what follows, also other quantities (e.g. $R$, $r_0$, $r_1$, $r_2$, $r_3$, $\phi$, $\xi_1$, $\xi_2$) will depend on $a$ and $b$, but this dependence will be omitted from the notation.

\begin{lemma}\label{exist}
Given any $a,b\in \R{}$, the solution $u$ to the ODE \eqref{ODE} exists for all times.
\end{lemma}

\begin{proof}
Applying \eqref{div2} to $f=\Delta^2 u$, and observing that $\Delta(\Delta ^2 u)=-e^{6u}\le 0$ we get
\begin{equation}\label{stima11}
\Delta^2 u(t)\le \Delta^2 u(s) \le \Delta^2 u(0)=16 b\quad 0\le s\le t,
\end{equation}
i.e. $\Delta^2 u(r)$ is monotone decreasing.
This and \eqref{div2} applied to $\Delta u$ yield
$$\Delta u(r)\le \Delta u(0)+\int_0^r\frac{1}{\omega_5 \rho^5}\int_{B_\rho}16b dx d\rho=6a+\int_0^r \frac{8}{3}b\rho d\rho=6a+\frac{4}{3}br^2. $$
A further application of \eqref{div2} to $u$ finally gives
\begin{equation}\label{stima7}
u(r)\le \int_0^r\frac{1}{\omega_5 \rho^5}\int_{B_\rho}(6a+\frac{4}{3}b|x|^2)dxd\rho
=\int_0^r (a\rho+\frac{\rho^3b}{6})d\rho =\frac{a}{2}r^2+\frac{b}{24}r^4=:\phi(r).
\end{equation}
Similar lower bounds can be obtained by observing that $-e^{6u}\ge -1$ for $u\le 0$. This proves that $u(r)$ cannot blow-up in finite time and, by standard ODE theory, $u(r)$ exists for every $r\ge 0$.
\end{proof}

\noindent\emph{Proof of \eqref{lima} (completed)}. Fix $b=3$ and take $a<0$.
The function $\phi(r)=\frac{a}{2}r^2+\frac{1}{8}r^4$
vanishes for $r=R=R(a):=2\sqrt{-a}$. In order to prove \eqref{lima} we shall investigate the behavior of $u$ in a neighborhood of $R$. The heuristic idea is that
$$u^{(j)}(0)=\phi^{(j)}(0),\quad \text{for }0\le j\le 5,\quad \quad \Delta^3 \phi\equiv 0, $$
and for every $\ve>0$ on $[\ve,R-\ve]$ we have $\phi \le C_\ve a \to -\infty$ and $|\Delta^3 u| \le e^{C_\ve a}\to 0$ as $a\to-\infty$, hence for $r\in [0,R-\ve]$ we expect $u(r)$ to be very close to $\phi(r)$. On the other hand, $u$ cannot stay close to $\phi$ for $r$ much larger than $R$ because eventually $-\Delta^3 u(r)$ will be large enough to make $\Delta^2 u$, $\Delta u$ and $u$ negative according to \eqref{div2} (see Fig. 1). Then it is crucial to show that $u$ stays close to $\phi$ for some $r>R$ (hence in a region where $\phi$ is positive and $\Delta^3 u$ is not necessarily small) and long enough to make the second integral in \eqref{lima} blow up as $a\to -\infty$.

\begin{figure}
\begin{center}
  \includegraphics[scale=.5, trim= 0cm 12cm 0cm 7cm, clip=true]{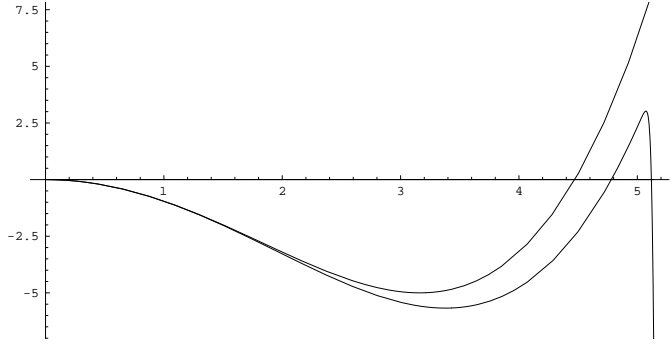}\label{fig1}
\caption{The functions $\phi(r)= \frac{a}{2}r^2+\frac{1}{8}r^4 $ (above) and $u_{a,3}(r)\le \phi(r)$. }
\end{center}
\end{figure}

\medskip
\noindent\emph{Step 1: Estimates of $u(R)$, $\Delta u(R)$ and $\Delta^2u(R)$.}
From \eqref{stima7} we infer
$$\Delta^3 u=-e^{6u}\ge -e^{6\phi},$$
which, together with \eqref{div2}, gives
\begin{equation}\label{stima1}
\Delta^2 u(r)=\Delta^2 u(0)+\int_0^r\frac{1}{\omega_5 \rho^5}\int_{B_\rho}\Delta^3 u dxd\rho
\ge 48 - \int_0^r\frac{1}{\omega_5 \rho^5}\int_{B_\rho}e^{6\phi(|x|)} dxd\rho.
\end{equation}
We can explicitly compute (see Lemma \ref{int} below and simplify \eqref{int2} using that $\phi(R)=0$ and $\int_{\sqrt{3}a}^{-\sqrt{3}a}e^{t^2}dt=2\int_{0}^{-\sqrt{3}a}e^{t^2}dt$)
$$\int_0^R\frac{1}{\omega_5 \rho^5}\int_{B_\rho}e^{6\phi(|x|)} dxd\rho= \frac{1}{48a}+\frac{(18a^2+1)\sqrt{3}}{144 a^2  }e^{-3 a^2} \int_0^{-\sqrt{3} a}e^{t^2}dt.$$
Then by \eqref{stima11} and Lemma \ref{lim} below we conclude that
\begin{equation}\label{stima16}
\Delta^2 u(r)\ge \Delta^2 u(R)\ge 48(1+O(a^{-1}))\quad \text{ for }0 \le r\le R=2\sqrt{-a}.
\end{equation}
where here and in the following $|a^kO(a^{-k})|\le C=C(k)$ as $a\to -\infty$ for every $k\in\R{}$.
Then applying \eqref{div2} as before we also obtain
$$\Delta u(r)\ge 6a+4(1+O(a^{-1})) r^2\quad \text{ for }0 \le r\le R$$
and
$$u(r)\ge \frac{a}{2}r^2+\frac{1+O(a^{-1})}{8}r^4=\phi(r)+O(a^{-1})r^4\quad \text{ for }0 \le r\le R. $$
At $r=R$ this reduces to
$$u(R)\ge O(a).$$

\medskip
\noindent\emph{Step 2: Behavior of $u(r)$, $\Delta u(r)$, $\Delta^2 u(r)$ for $r\ge R$.}
Define $r_0$ (depending on $a<0$) as
$$r_0:=\inf\{r>0: u(r)=0\} \in [R,\infty].$$
We first claim that $r_0<\infty$.
We have by Lemma \ref{int} and Lemma \ref{lim}
\begin{equation}\label{stima13}
\int_{B_{R}}e^{6\phi} dx=\omega_5\bigg(-\frac{4 a}{3}+\frac{4(6a^2-1)\sqrt{3}}{9}e^{-3 a^2} \int_0^{-\sqrt{3}a} e^{t^2}dt \bigg)=O(a).
\end{equation}
Since on $B_{r_0}$ we have $u\le 0$, hence $\Delta^3 u\ge -1$, using \eqref{div1}-\eqref{div2} and \eqref{stima13} we get for $r\in [R,r_0]$
\begin{equation}\label{stima18}
\begin{split}
\Delta^2 u(r)&\ge \Delta^2 u(R)-\int_{R}^{r}\frac{1}{\omega_5 \rho^5}\bigg(\int_{B_{R}} e^{6\phi}dx+\int_{B_\rho\setminus B_{R}} 1dx\bigg)d\rho\\
&\ge  48+O(a)\bigg[\frac{1}{R^4}-\frac{1}{r^4}\bigg]-\int_R^r\frac{\rho^6-R^6}{6\rho^5}d\rho
\end{split}
\end{equation}
Assuming $r\in [R, 2R]$ we can now bound with a Taylor expansion
\begin{equation}\label{stima14}
\frac{1}{R^4}-\frac{1}{r^4}=R^{-4}\tilde O\Big(\frac{r-R}{R}\Big)
\end{equation}
and
$$\rho^6-R^6\le r^6-R^6= R^6 \tilde O\Big(\frac{r-R}{R}\Big),\quad \text{for }\rho\in [R,r],$$
which together with \eqref{stima14} yields
\begin{equation}\label{stima15}
\int_R^r\frac{\rho^6-R^6}{6\rho^5}d\rho\le \int_R^r\frac{r^6-R^6}{6\rho^5}d\rho\le R^2\tilde O\Big(\Big(\frac{r-R}{R}\Big)^2\Big),
\end{equation}
where for any $k\in\R{}$ we have $|t^{-k}\tilde O(t^k)|\le C=C(k)$ uniformly for $0\le t\le 1$.
Using \eqref{stima14} and \eqref{stima15} we bound in \eqref{stima18}
$$\Delta^2 u(r)\ge 48+O(a^{-1})\tilde O\Big(\frac{r-R}{R}\Big) +R^2\tilde O\Big(\Big(\frac{r-R}{R}\Big)^2\Big),\quad r\in [R, \min\{r_0,2R\}],$$
whence
$$\Delta^2 u(r)\ge 48+O(a^{-1}) +R^2\tilde O\Big(\Big(\frac{r-R}{R}\Big)^2\Big)\chi_{(R,\infty)}(r),\quad r\in [0,\min \{r_0,2R\}],$$
where $\chi_{(R,\infty)}(r)=0$ for $r\in [0,R]$ and $\chi_{(R,\infty)}(r)=1$ for $r>R$.
Then with \eqref{div2} we estimate for $r\in [0,\min \{r_0,2R\}]$
\begin{equation}\label{stima3}
\begin{split}
\Delta u(r)&\ge 6a+4(1+O(a^{-1}))r^2+ \chi_{(R,\infty)}(r)\int_{R}^r\frac{1}{\omega_5\rho^5}\int_{B_\rho\setminus B_R}R^2\tilde O\Big(\Big(\frac{|x|-R}{R}\Big)^2\Big)dxd\rho\\
&=6a+4(1+O(a^{-1}))r^2+R^4\tilde O\Big(\Big(\frac{r-R}{R}\Big)^4\Big)\chi_{(R,\infty)}(r).
\end{split}
\end{equation}
and
\begin{equation}\label{stima4}
\begin{split}
u(r)&\ge \frac{a}{2}r^2+\frac{1+O(a^{-1})}{8}r^4+\chi_{(R,\infty)}(r)\int_{R}^r\frac{1}{\omega_5\rho^5}\int_{B_\rho\setminus B_R}R^4\tilde O\Big(\Big(\frac{|x|-R}{R}\Big)^4\Big) dxd\rho\\
&=\phi(r)+O(a^{-1})r^4+R^6\tilde O\Big(\Big(\frac{r-R}{R}\Big)^6\Big)\chi_{(R,\infty)}(r),
\end{split}
\end{equation}
where the integrals in \eqref{stima3} and \eqref{stima4} are easily estimated bounding $|x|$ with $r$ and applying \eqref{stima15}.

Making a Taylor expansion of $\phi(r)$ at $r=R$ and using that $\phi(R)=0$, we can further estimate the right-hand side of \eqref{stima4} for $r\in [R,\min\{r_0,2R\}]$ as
\begin{equation*}
\begin{split}
u(r)&\ge \phi'(R)(r-R) + R^2 \tilde O\Big(\Big(\frac{r-R}{R}\Big)^2\Big)+O(a^{-1})r^4+R^6\tilde O\Big(\Big(\frac{r-R}{R}\Big)^6\Big)\\
&=-a R(r-R) +O(a^{-1})R^4+ R^2 \tilde O\Big(\Big(\frac{r-R}{R}\Big)^2\Big)+R^6\tilde O\Big(\Big(\frac{r-R}{R}\Big)^6\Big)=:\psi_a(r)
\end{split}
\end{equation*}
Now choosing $r=R(1+1/\sqrt{-a})$, so that $(r-R)/R\to 0$ as $a\to-\infty$, we get
\begin{equation*}
\lim_{a\to -\infty}\psi_a(R(1+1/\sqrt{-a})) \ge \lim_{a\to -\infty}\bigg( 4(-a)^\frac{3}{2} +O(a)-C\bigg)=\infty.
\end{equation*}
In particular
$$r_0\in [R, R(1+1/\sqrt{-a})].$$
We now claim that
\begin{equation}\label{stima2}
\lim_{a\to-\infty}\Delta u(r_0)=\infty.
\end{equation}
Indeed we infer from \eqref{stima3}
\begin{equation*}
\begin{split}
\Delta u(r_0)&\ge 6a+4(1+O(a^{-1})) r_0^2-C\\
&\ge 6a+4(1+O(a^{-1}))R^2-C\\
&\ge -10a-C,
\end{split}
\end{equation*}
for $-a$ large enough, whence \eqref{stima2}. 
Set
$$r_1=r_1(a):=\inf\{r>r_0: u(r)=0\}.$$
Applying \eqref{div1} to \eqref{stima3}, and recalling that $\frac{r_0-R}{R}\le\frac{1}{\sqrt{a}}$, similar to \eqref{stima4} we obtain
\begin{equation*}
\begin{split}
u'(r_0)&\ge ar_0+\frac{1+O(a^{-1})}{2}r_0^3-C\\
&\ge ar_0 +\frac{1+O(a^{-1})}{2}r_0 R^2 -C\\
&\ge -ar_0 -C.
\end{split}
\end{equation*}
In particular for $-a$ large enough we have $u'(r_0)>0$, which implies $r_1>r_0$. Using \eqref{div1}-\eqref{div2} and that $\Delta^3 u(r)\le -1$ for $r\in [r_0,r_1]$, it is not difficult to see that $r_1<\infty$. 
Moreover there exists at least a point $r_2=r_2(a)\in (r_0,r_1]$ such that $u'(r_2)\le 0$, which in turn implies that
\begin{equation}\label{stima5}
\Delta u(r_3)<0\quad \text{for some }r_3=r_3(a)\in (r_0,r_2],
\end{equation}
since otherwise we would have by \eqref{div1}
$$u'(r_2)=\frac{1}{\omega_5r_2^5}\int_{B_{r_0}}\Delta u dx+\frac{1}{\omega_5r_2^5}\int_{B_{r_2}\setminus B_{r_0}}\Delta u dx\ge \frac{r_0^5}{r_2^5}u'(r_0)>0, $$
contradiction.

\medskip

\noindent\emph{Step 3: Conclusion.} We now use the estimates obtained in Step 1 and Step 2 to prove \eqref{lima}. 

From \eqref{div2}, \eqref{stima2} and \eqref{stima5} we infer
\begin{equation}\label{stima8}
\lim_{a\to-\infty}\int_{r_0}^{r_3}\frac{1}{\omega_5 r^5}\int_{B_r}\Delta^2 u dx dr=\lim_{a\to-\infty}(\Delta u(r_3)-\Delta u(r_0))=-\infty,
\end{equation}
hence by the monotonicity of $\Delta^2 u(r)$ (see \eqref{stima11})
\begin{equation}\label{stima12}
\lim_{a\to -\infty}\Delta^2 u(r_3)(r_3^2-r_0^2)=-\infty.
\end{equation}
We now claim that
\begin{equation}\label{stima6}
\lim_{a\to-\infty}\int_{B_{r_3}} e^{6u}dx=\infty.
\end{equation}
Indeed consider on the contrary an arbitrary sequence $a_k$ with $\lim_{k\to\infty}a_k=-\infty$ and
\begin{equation}\label{stima19}
\lim_{k\to \to \infty}\int_{B_{r_3}} e^{6u}dx<\infty,
\end{equation}
where here $r_3$ and $u$ depend on $a_k$ instead of $a$ of course.
Since $u\ge 0$ in $B_{r_3}\setminus B_{r_0}$ we have
$$\int_{B_{r_3}} e^{6u}dx \ge  \int_{B_{r_3}\setminus B_{r_0}} 1dx  =\frac{\omega_5}{6}(r_3^6-r_0^6).$$
Now observe that 
$(r_3^6-r_0^6)\ge (r_3^2-r_0^2)r_0^4$
to conclude that \eqref{stima19} implies
\begin{equation}\label{r0r1}
\lim_{k\to \infty}(r_3^2-r_0^2)\le \lim_{k\to\infty}\frac{r_3^6-r_0^6}{r_0^4}= 0.
\end{equation}
Then \eqref{div2}, \eqref{stima16} and \eqref{stima12} yield
\begin{equation*}
\begin{split}
(r_3^2-r_0^2)\int_{R}^{r_3}\frac{1}{\omega_5 r^5}\int_{B_r}e^{6 u}dxdr& =(r_3^2-r_0^2)(\Delta^2 u(R)-\Delta^2 u(r_3))\\
&\ge  -\Delta^2 u(r_3)(r_3^2-r_0^2) \to\infty\quad \text{as }k\to \infty.
\end{split}
\end{equation*}
\medskip
By \eqref{r0r1} we also have
$$\lim_{k\to \infty}\int_{R}^{r_3}\frac{1}{\omega_5 r^5}\int_{B_r}e^{6 u}dxdr =\infty,$$
which implies at once
$$\lim_{k\to \infty}\int_{B_{r_3}} e^{6u}dx\ge \lim_{k\to \infty}4R^4\omega_5 \int_{R}^{r_3}\frac{1}{\omega_5 r^5}\int_{B_{r_3}}e^{6 u}dxdr =\infty,$$
contradicting \eqref{stima19}. Then \eqref{stima6} is proven.

It remains to show that
$$\int_{\R{6}}e^{6u}dx<\infty,$$
at least for $-a$ large enough. 
It follows from \eqref{stima12} and the monotonicity of $\Delta^2 u$ that for $-a$ large enough we have
\begin{equation}\label{ra0b}
\Delta^2 u(r)<B< 0,\quad \text{for }r\ge r_3,
\end{equation}
and,  using \eqref{div1}-\eqref{div2} as already done several times, we can find $r_a\ge r_3$ such that
\begin{equation}\label{rab}
\begin{split}
(\Delta u)'(r)&<\frac{B}{6}r,\quad \Delta u(r)<\frac{B}{12}r^2, \quad u'(r)<\frac{B}{96}r^3,\quad 
u(r)<\frac{B}{384}r^4,\quad \text{for }r\ge r_a.
\end{split}
\end{equation}
Then
$$\int_{\R{6}}e^{6u}dx \le \int_{B_{r_a}}e^{6u}dx+\int_{\R{6}\setminus B_{r_a}}e^{\frac{B}{64}|x|^2}dx <\infty,$$
as wished.  \hfill $\square$

\subsection{Two useful lemmas}

We now state and prove two lemmas used in the proof of Theorem \ref{trm1}.

\begin{lemma}\label{int} For $\phi(r)=\frac{a}{2}r^2+\frac{1}{8}r^4$, $a\le 0$, we have
\begin{equation}\label{int1}
\int_{B_r}e^{6\phi(|x|)}dx =\omega_5\bigg[\frac{2}{3}a+\frac{1}{3}e^{6\phi(r)}(-2a+r^2)+\frac{(12a^2-2)\sqrt{3}}{9}e^{-3a^2}\int^{-\sqrt{3}a}_{-\sqrt{3}(a+r^2/2)} e^{t^2}dt\bigg]=:\xi_1(r)
\end{equation}
and
\begin{equation}\label{int2}
\begin{split}
\int_0^r\frac{1}{\omega_5\rho^5}\int_{B_\rho}e^{6\phi(|x|)}dxd\rho=&\frac{-2a-e^{6\phi(r)}(-2a+r^2)}{12r^4}\\
& +\frac{(2-12a^2+3r^4)\sqrt{3}}{36r^4}e^{-3a^2}  \int^{-\sqrt{3}a}_{-\sqrt{3}(a+r^2/2)} e^{t^2}dt :=\xi_2(r) 
\end{split}
\end{equation}
\end{lemma}

\begin{proof}
Patiently differentiating, using that $e^{-3a^2}\frac{d}{dr} \int^{-\sqrt{3}a}_{-\sqrt{3}(a+r^2/2))} e^{t^2}dt=\sqrt{3} r e^{6\phi(r)}$, one sees that
$$\xi_1'(r)=\omega_5 r^5e^{6\phi(r)},\quad \xi_2'(r)=\frac{\xi_1(r)}{\omega_5r^5}.$$
Using that $\phi(0)=0$ it is also easy to see that $\xi_1(0)=0$.

Since $\xi_2(0)$ is not defined, we will compute the limit of $\xi_2(r)$ as $r\to 0$. We first compute the Taylor expansions
$$e^{6\phi(r)}=1+3ar^2+\frac{3}{4}(1+6a^2)r^4+r^4o(1),$$
and
$$\sqrt{3}e^{-3a^2}  \int^{-\sqrt{3}a}_{-\sqrt{3}(a+r^2/2)} e^{t^2}dt  =\frac{3}{2}r^2+\frac{9}{4}ar^4+r^4o(1), $$
with errors $o(1)\to 0$ as $r\to 0$. Then
\begin{equation*}
\begin{split}
\frac{-2a-e^{6\phi(r)}(-2a+r^2)}{12r^4}&=\frac{(1-6a^2) r^2+(\frac{3}{2}a-9a^3)r^4}{12r^4}+o(1)\\
&=-\frac{(2-12a^2+3r^4)\sqrt{3}}{36r^4}e^{-3a^2}  \int^{-\sqrt{3}a}_{-\sqrt{3}(a+r^2/2)} e^{t^2}dt,
\end{split}
\end{equation*}
with $o(1)\to 0$ as $r\to 0$. Hence $\lim_{r\to 0}\xi_2(r)=0$.
\end{proof}

\begin{lemma}\label{lim} We have
\begin{equation}\label{eqint}
\lim_{r\to\infty} re^{-r^2}\int_0^r e^{t^2}dt=\frac{1}{2}.
\end{equation}
\end{lemma}

\begin{proof} Clearly \eqref{eqint} is equivalent to
\begin{equation}\label{eqint2}
\lim_{r\to\infty} re^{-r^2}\int_2^r e^{t^2}dt=\frac{1}{2}.
\end{equation}
Integrating by parts we get for $r\ge 2$
\begin{equation}\label{eqint3}
re^{-r^2}\int_2^r e^{t^2}dt=\frac{1}{2}-\frac{re^{-r^2+4}}{4}+re^{-r^2}\int_2^r \frac{e^{t^2}}{2t^2}dt.
\end{equation}
Another integration by parts yields
$$re^{-r^2}\int_2^r \frac{e^{t^2}}{2t^2}dt=\frac{1}{4r^2}-\frac{re^{-r^2+4}}{32}+re^{-r^2}\int_2^r\frac{e^{t^2}}{12t^4}dt\to 0 \quad \text{as }r\to \infty,$$
where we used that the function $t^{-4}e^{t^2}$ is increasing on $[2,r]$, hence
$$0\le \int_2^r \frac{e^{t^2}}{12t^4}dt\le \int_2^r \frac{e^{r^2}}{12r^4}dt= (r-2)\frac{e^{r^2}}{12r^4}.$$
We conclude by taking the limit as $r\to \infty$ in \eqref{eqint3}.
\end{proof}

\section{Proof of Theorem \ref{trm0}}

We start with the following lemma.

\begin{lemma}\label{cont} Set
$$V(a)=\frac{1}{5!}\int_{\R{6}} e^{6u_{a,3}}dx$$
where $u=u_{a,3}$ is the solution to \eqref{ODE} for given $a<0$ and $b=3$. Then there exists $a^*<0$ such that $V$ is continuous on $( -\infty,a^*]$.
\end{lemma}

\begin{proof} It follows from \eqref{stima8} and the monotonicity of $\Delta^2 u$ that we can fix $-a^*$ so large that
$$\lim_{r\to \infty}\Delta^2 u_{a,3}(r)<0,\quad \text{for every }a\le a^*.$$ 
Fix now $\ve>0$. Given $a\le a^*$ it is not difficult to find $r_a>0$ and $B=B(a)<0$ such that
\begin{equation}\label{ra0}
\Delta^2 u_{a,3}(r)< B< 0,\quad \text{for }r\ge r_a
\end{equation}
and, possibly choosing $r_a$ larger,  using \eqref{div1}-\eqref{div2} as already done in the proof of Theorem \ref{trm1}, we get
\begin{equation}\label{ra}
\begin{split}
(\Delta u_{a,3})'(r)&<\frac{B}{6}r,\quad \Delta u_{a,3}(r)<\frac{B}{12}r^2, \quad u_{a,3}'(r)<\frac{B}{96}r^3,\quad 
u_{a,3}(r)<\frac{B}{384}r^4,\quad \text{for }r\ge r_a.
\end{split}
\end{equation}
By possibly choosing $r_a$ even larger we can also assume that
\begin{equation}\label{stima9}
\int_{\R{6}\setminus B_{r_a}}e^{\frac{B}{64}|x|^4}dx< \frac{\ve}{2}.
\end{equation}
By ODE theory the solution $u_{a,3}$ to \eqref{ODE} is continuous with respect to $a$ in $C^k_{\loc}(\R{6})$ for every $k\ge 0$, in the sense that for any $r'>0$, $u_{a',3}\to u_{a,3}$ in $C^k(B_{r'})$ as $a'\to a$.
In particular we can find $\delta>0$ (depending on $\ve$) such that if $|a-a'|<\delta$ then \eqref{ra0}-\eqref{ra} with $a$ replaced by $a'$ are still satisfied for $r=r_a$ (not $r_{a'}$) and \eqref{ra0} holds also for every $r>r_a$ since $\Delta^2 u_{a',3}(r)$ is decreasing in $r$ (see \eqref{stima11}). Then, with \eqref{div1}-\eqref{div2} we can also get the bounds in \eqref{ra} for every $r\ge r_a$ (and $u_{a',3}$ instead of $u_{a,3}$). For instance
\begin{equation*}
\begin{split}
(\Delta u_{a',3})'(r)&=\frac{1}{\omega_5r^5}\int_{B_r}\Delta^3 u_{a',3} dx =\bigg(\frac{r_{a}}{r}\bigg)^5 (\Delta u_{a',3})'(r_{a})+\frac{1}{\omega_5 r^5}\int_{B_r\setminus B_{r_{a}}} \Delta^2 u_{a',3}dx\\
&< \bigg(\frac{r_{a}}{r}\bigg)^5 \frac{Br_{a}}{6} +\frac{B(r^6-r_{a}^6)}{6r^5}=\frac{B}{6}r.
\end{split}
\end{equation*}
Furthermore, up to taking $\delta>0$ even smaller, we can assume that
\begin{equation}\label{stima10}
\bigg|\int_{B_{r_a}} e^{6u_{a',3}}dx -\int_{B_{r_a}} e^{6u_{a,3}}dx\bigg| < \frac{\ve}{2}.
\end{equation}
Finally, the last bound in \eqref{ra} and \eqref{stima9} imply at once
$$\bigg|\int_{\R{6}\setminus B_{r_a}} e^{6u_{a',3}}dx -\int_{\R{6}\setminus B_{r_a}} e^{6u_{a,3}}dx\bigg| < \frac{\ve}{2},$$
which together with \eqref{stima10} completes the proof.
\end{proof}

\medskip
\noindent\emph{Proof of Theorem \ref{trm0} (completed).}
Set $V^*=V(a^*)$, where $a^*$ is given by Lemma \ref{cont}. By Lemma \ref{cont}, Theorem \ref{trm1} and the intermediate value theorem, for every $V \ge V^*$ there exists $a\le a^*$ such that
$$\frac{1}{5!}\int_{\R{6}}e^{6u_{a,3}} dx=V,$$
hence the metric $g_{u_{a,3}}= e^{2u_{a,3}}|dx|^2$ has constant $Q$-curvature equal to $1$ and $\vol(g_{u_{a,3}})=5! V$.
Applying the transformation 
$$u=u_{a,3}-\frac{1}{6}\log 5!$$
it follows at once that the metric $g_u=e^{2u}|dx|^2$ satisfies $\vol(g_u)=V$ and
$Q_{g_u}\equiv  5!$, hence $u$ solves \eqref{eq0}-\eqref{area}.
\hfill $\square$

\section{Proof of Theorems \ref{trm3} and \ref{trm4}}

When $f:\R{n}\to \R{}$ is radially symmetric we have $\Delta f(x)=f''(|x|)+\frac{n-1}{|x|}f'(|x|)$. In particular we have
\begin{equation}\label{lapl}
\Delta^m r^{2m}=2^{2m}m(2m-1)!\quad \text{in }\R{2m}.
\end{equation}
For $m\ge 2$ and $b\le 0$ let $u_b$ solve
$$
\left\{
\begin{array}{ll}
\Delta^{m} u_b= -(2m-1)!e^{2mu_b}& \text{in }\R{2m}\\
u_b^{(j)}(0)=0 &\text{for } 0\leq j\leq 2m-1,\;j\ne 2m-2\\
u_b^{(2m-2)}=b.&
\end{array}
\right.
$$
From \eqref{div1}-\eqref{div2} it follows that $u_0\leq 0$, hence $\Delta^m u_0\geq -(2m-1)!$. We claim that
$$u_0(r)\ge \psi(r):=-\frac{r^{2m}}{2^{2m}m}.$$
Indeed according to \eqref{lapl} $\psi$ solves
$$
\Delta^{m} \psi= -(2m-1)!\le \Delta^m u_0 \quad  \text{in }\R{2m}
$$
and
$$\psi^{(j)}(0)=0 = u_0^{(j)}(0)\quad\text{for }0\le j\le 2m-1,$$
which implies
$$\Delta^{j}\psi(0)=0=\Delta^{j}u_0(0)\quad \text{for }0\le j\le m-1,$$
see \cite[Lemma 17]{mar1}. Then the claim follows from \eqref{div1}-\eqref{div2} and a simple induction.

Now integrating we get
\begin{equation*}
\begin{split}
\int_{\R{2m}}e^{2mu_0}dx&\geq \int_{\R{2m}}e^{2m\psi}dx =\omega_{2m-1}\int_0^\infty r^{2m-1}\exp\Big(-\frac{r^{2m}}{2^{2m-1}}\Big)dr =\frac{2^{2m-2}\omega_{2m-1}}{m}=:V_m.
\end{split}
\end{equation*}
Using the formulas
$$\omega_{2m-1}=\vol(S^{2m-1})=\frac{2\pi^m}{(m-1)!},\quad \omega_{2m}=\vol(S^{2m})=\frac{2^{2m}(m-1)!\pi^m}{(2m-1)!},\quad m\ge 1$$
we verify
\begin{equation}\label{Vm0}
V_m= \frac{(2m)!}{4(m!)^2}\omega_{2m},\quad \frac{V_2}{\omega_4}=\frac{3}{2}>1, \quad \frac{V_{m+1}}{\omega_{2m+2}}\Big(\frac{V_m}{\omega_{2m}}\Big)^{-1}=\frac{(2m+2)(2m+1)}{(m+1)^2}>1,
\end{equation}
hence by induction
\begin{equation}\label{Vm}
V_m>\vol(S^{2m})\quad \text{for }m\ge 2.
\end{equation}

\medskip

With the same argument used to prove Lemma \ref{cont} we can show that the function
$$V(b):=\int_{\R{2m}}e^{6u_b}dx,\quad b\in (-\infty,0]$$
is finite and continuous. Indeed it is enough to replace \eqref{ra0} with
$$\Delta^{m-1}u_b(r)\le B<0\quad \text{for }r\ge r_b, $$
and \eqref{ra} with
$$(\Delta^{m-1-j}u_b)'(r)<C_{m,j}Br^{2j-1},\quad  \Delta^{m-1-j}u_b(r)<D_{m,j}Br^{2j},\quad \text{for }r\ge r_b, \; 1\le j\le m-1$$
where $r_b$ is chosen large enough and
$$C_{m,1}=\frac{1}{2m},\quad D_{m,j}=\frac{C_{m,j}}{2j},\quad C_{m,j+1}=\frac{D_{m,j}}{2m+2j},$$
whence
$$C_{m,j}=\frac{(m-1)!}{2^{2j-1}(j-1)!(m+j-1)!},\quad D_{m,j}=\frac{(m-1)!}{2^{2j}j!(m+j-1)!}.$$
Moreover, using that $\Delta^{m-1}u_b(0)=C_m b$ for some constant $C_m>0$,  $\Delta^m u_b(r)\le 0$ for $r\ge 0$ and \eqref{div1}-\eqref{div2} as before,
we easily obtain
\begin{equation}\label{ub}
u_b(r)\le E_m br^{2m-2},
\end{equation}
where $E_m:=C_m C_{m,m-1}>0$, hence
$$\lim_{b\to -\infty}V(b)\le \lim_{b\to -\infty}\int_{\R{6}}e^{6E_m b |x|^{2m-2}}dx=0.$$
By continuity we conclude that for every $V\in (0,V_m]$ there exists $b\le 0$ such that $u=u_b$ solves \eqref{eq0}-\eqref{area} if $m$ is odd or \eqref{eq0-}-\eqref{area} if $m$ is even. Taking \eqref{Vm} into account it only remains to prove that the solutions $u_b$ corresponding to $V=\vol(S^{2m})$ is not a spherical one. This follows immediately from \eqref{ub}, which is not compatible with \eqref{stand}.
\hfill $\square$

\section{Applications and open questions}

\paragraph{Possible gap phenomenon}

Theorems \ref{trm0} and \ref{trm3} guarantee that for $m=3$  there exists a solution to \eqref{eq0}-\eqref{area} for every $V\in (0,V_3]\cup [V^*,\infty)$, with possibly $V_3<V^*$.
Could it be that for some $V\in (V_3,V^*)$ Problem \eqref{eq0}-\eqref{area} admits no solution?

If we restrict to rotationally symmetric solutions, some heuristic arguments show that the volume of a solution to \eqref{ODE}, i.e. the function
$$V(a,b):=\int_{\R{6}} e^{6u_{a,b}(|x|)}dx$$
need not be continuous for all $(a,b)\in \R{2}$, hence the image of the function $V$ might not be connected.

\medskip

\paragraph{Higher dimensions and negative curvature}

It is natural to ask whether Theorems \ref{trm0} and \ref{trm1} generalize to the case $m>3$ or whether an analogous statement holds when $m\ge 2$ and \eqref{eq0-} is considered instead of \eqref{eq0}. Since the sign on the right-hand side of the ODE \eqref{ODE} plays a crucial role, we would expect that part of the proof of Theorem \ref{trm1} can be recycled for $\eqref{eq0}$ when $m\ge 5$ is odd, or for \eqref{eq0-} when $m$ is even.

For instance let $u_a=u_a(r)$ be the solution in $\R{4}$ of
$$
\left\{
\begin{array}{l}
\Delta^2 u_a= -6e^{4u_a}\\
u_a(0)=u'_a(0)=u'''_a(0)=0\\
u''_a(0)=a.
\end{array}
\right.
$$
It should not be difficult to see that $u_a(r)$ exists for all $r\ge 0$ and that $\int_{\R{4}}e^{4u_a(|x|)}dx<\infty$.
Do we also have
$$\lim_{a\to+\infty}\int_{\R{4}}e^{4u_a(|x|)}dx=\infty?$$

\paragraph{Non-radial solutions}

The proof of Theorem $C$ cannot be extended to provide non-radial solutions to \eqref{eq0}-\eqref{area} for $m\ge 3$ and  $V\ge \vol(S^{2m})$, but it is natural to conjecture that they do exist.

\paragraph{Concentration phenomena}

The classification results of the solutions to \eqref{eq0}-\eqref{area}, \cite{CL}, \cite{Lin}, \cite{Xu} and \cite{mar1},  have been used to understand the asymptotic behavior of unbounded sequences of solutions to the prescribed Gaussian curvature problem on $2$-dimensional domains (see e.g. \cite{BM} and \cite{LS}), on $S^2$ (see \cite{str1}) and to the prescribed $Q$-curvature equation in dimension $2m$ (see e.g. \cite{DR}, \cite{mal}, \cite{MS}, \cite{ndi}, \cite{rob1}, \cite{rob2}, \cite{mar3}, \cite{mar4}).

For instance consider the following model problem.
Let $\Omega\subset\R{2m}$ be a connected open set and consider a sequence $(u_k)$ of solutions to the equation

\begin{equation}\label{eq1}
(-\Delta)^m u_k =Q_k e^{2mu_k}\quad \text{in }\Omega,
\end{equation}
where
\begin{equation}\label{Vk}
Q_k\to Q_0 \quad \text{in } C^1_{\loc}(\Omega),\quad \limsup_{k\to\infty}\int_{\Omega} e^{2mu_k}dx <\infty,
\end{equation}
with the following interpretation: $g_k:=e^{2u_k}|dx|^2$ is a sequence of conformal metrics on $\Omega$ with $Q$-curvatures $Q_{g_k}=Q_k$ and equibounded volumes.

As shown in \cite{ARS} unbounded sequences of solutions to \eqref{eq1}-\eqref{Vk} can exhibit pathological behaviors in dimension $4$ (and higher), contrary to the elegant results of \cite{BM} and \cite{LS} in dimension $2$. This is partly due to Theorem A. In fact for $m\ge 2$ and $\alpha\in (0,(2m-1)!\vol(S^{2m})]$ one can found a sequence $(u_k)$ of solutions to \eqref{eq1}-\eqref{Vk} with $Q_0>0$ and
\begin{equation}\label{alpha}
\lim_{R\to 0} \lim_{k\to\infty} \int_{B_{R}(x_0)}|Q_k| e^{2mu_k}dx=\alpha\quad \text{for some }x_0\in\Omega.
\end{equation}
For $m=2$ this was made very precise by F. Robert \cite{rob1} in the radially symmetric case. In higher dimension or when $Q_0$ is not necessarily positive, thanks to Theorems \ref{trm0}-\ref{trm4} we see that $\alpha$ can take values larger than $(2m-1)!\vol(S^{2m})$. Indeed if $u$ is a solution to \eqref{eq0}-\eqref{area} or \eqref{eq0-}-\eqref{area}, then $u_k:=u(kx)+\log k$ satisfies \eqref{eq1}-\eqref{Vk} with $\Omega=\R{2m}$, $Q_k\equiv \pm (2m-1)!$ and
$$|Q_k| e^{2mu_k}dx \rightharpoondown (2m-1)! V \delta_0,\quad \text{weakly as measures}.$$
When $m=2$, $Q_0>0$ (say $Q_0\equiv 6$) it is unclear whether one could have concentration points carrying more $Q$-curvature than $6\vol(S^4)$, i.e. whether one can take $\alpha>6\vol(S^4)$ in \eqref{alpha}. 
Theorem B suggests that if  the answer is affirmative, this should be due to the convergence to the same blow-up point of two or more blow-ups. Such a phenomenon is unknown in dimension $4$ and higher, but was shown in dimension $2$ by  Wang \cite{wan} with a technique which, based on the abundance of conformal transformations of $\mathbb{C}$ into itself, does not extend to higher dimensions.



\begin{thebibliography}{2}
\small
\bibitem[ARS]{ARS} \textsc{Adimurthi, F. Robert, M. Struwe} \emph{Concentration phenomena for Liouville's equation in dimension 4}, Journal EMS \textbf{8} (2006), 171-180.
\bibitem[BM]{BM} \textsc{H. Br\'ezis, F. Merle} \emph{Uniform estimates and blow-up behaviour for solutions of $-\Delta u=V(x)e^u$ in two dimensions}, Comm. Partial Differential Equations \textbf{16} (1991), 1223-1253.
\bibitem[Cha]{cha} \textsc{S-Y. A. Chang} \emph{Non-linear Elliptic Equations in Conformal Geometry}, Zurich lecture notes in advanced mathematics, EMS (2004).
\bibitem[CC]{CC} \textsc{S-Y. A. Chang, W. Chen} \emph{A note on a class of higher order conformally covariant equations}, Discrete Contin. Dynam. Systems \textbf{63} (2001), 275-281.
\bibitem[CY]{CY} \textsc{S-Y. A. Chang, P. Yang} \emph{On uniqueness of solutions of $n$-th order differential equations in conformal geometry}, Math. Res. Lett. \textbf{4} (1997), 91-102.
\bibitem[CL]{CL} \textsc{W. Chen, C. Li} \emph{Classification of solutions of some nonlinear elliptic equations}, Duke Math. J. \textbf{63} (3) (1991), 615-622.
\bibitem[DR]{DR} \textsc{O. Druet, F. Robert} \emph{Bubbling phenomena for fourth-order four-dimensional PDEs with exponential growth}, Proc. Amer. Math. Soc \textbf{3} (2006), 897-908.
\bibitem[FG]{FG} \textsc{C. Fefferman, C. R. Graham} \emph{$Q$-curvature and Poincar\'e metrics}, Math. Res. Lett. \textbf{9} (2002), 139-151.
\bibitem[FH]{FH} \textsc{C. Fefferman, K. Hirachi} \emph{Ambient metric construction of $Q$-curvature in conformal and CR geometry}, Math. Res. Lett. \textbf{10} (2003), 819-831.
\bibitem[LS]{LS} \textsc{Y. Li, I. Shafrir} \emph{Blow-up analysis for solutions of $-\Delta u=Ve^u$ in dimension two}, Indiana Univ. Math. J. \textbf{43} (1994), 1255-1270.
\bibitem[Lin]{Lin} \textsc{C. S. Lin} \emph{A classification of solutions of conformally invariant fourth order equations in $\R{n}$}, Comm. Math. Helv \textbf{73} (1998), 206-231.
\bibitem[Mal]{mal} \textsc{A. Malchiodi} \emph{Compactness of solutions to some geometric fourth-order equations}, J. reine angew. Math. \textbf{594} (2006), 137-174.
\bibitem[MS]{MS} \textsc{A. Malchiodi, M. Struwe} \emph{$Q$-curvature flow on $S^4$}, J. Diff. Geom. \textbf{73} (2006), 1-44.
\bibitem[Mar1]{mar1} \textsc{L. Martinazzi}  \emph{Classification of solutions to the higher order Liouville's equation on $\mathbb{R}^{2m}$}, Math. Z. \textbf{263} (2009), 307-329.
\bibitem[Mar2]{mar2} \textsc{L. Martinazzi}  \emph{Conformal metrics on $\mathbb{R}^{2m}$ with constant $Q$-curvature}, Rend. Lincei. Mat. Appl. \textbf{19} (2008), 279-292.
\bibitem[Mar3]{mar3} \textsc{L. Martinazzi} \emph{Concentration-compactness phenomena in higher order Liouville's equation}, J. Functional Anal. \textbf{256} (2009), 3743-3771
\bibitem[Mar4]{mar4} \textsc{L. Martinazzi} \emph{Quantization for the prescribed $Q$-curvature equation on open domains }, Commun. Contemp. Math. \textbf{13} (2011), 533-551.
\bibitem[Ndi]{ndi} \textsc{C. B. Ndiaye} \emph{Constant $Q$-curvature metrics in arbitrary dimension}, J. Func. Analysis \textbf{251} no.1 (2007), 1-58.
\bibitem[Rob1]{rob1} \textsc{F. Robert} \emph{Concentration phenomena for a fourth order equation with exponential growth: the radial case}, J. Differential Equations \textbf{231} (2006), 135-164.
\bibitem[Rob2]{rob2} \textsc{F. Robert} \emph{Quantization effects for a fourth order equation of exponential growth in dimension four}, Proc. Roy. Soc. Edinburgh Sec. A \textbf{137} (2007), 531-553.
\bibitem[Str4]{str1} \textsc{M. Struwe} \emph{A flow approach to Nirenberg's problem}, Duke Math. J. \textbf{128}(1) (2005), 19-64.
\bibitem[Wan]{wan} \textsc{S. Wang,} \emph{An example of a blow-up sequence for $-\Delta u=V(x)e^u$}, Differential Integral Equations \textbf{5} (1992), 1111-1114.
\bibitem[WX]{WX} \textsc{J. Wei, X-W. Xu} \emph{Classification of solutions of higher order conformally invariant equations}, Math. Ann \textbf{313} (1999), 207-228.
\bibitem[WY]{WY} \textsc{J. Wei, D. Ye} \emph{Nonradial solutions for a conformally invariant fourth order equation in $\R{4}$}, preprint (2006).
\bibitem[Xu]{Xu} \textsc{X-W. Xu} \emph{Uniqueness theorems for integral equations and its application}, J. Funct. Anal. \textbf{247} (2007), no. 1, 95-109.
\end{thebibliography}
\end{document}